\documentclass{amsart}
\usepackage{amsmath,amssymb,amsthm,amscd}

\newcommand{\Z}{{\mathbb Z}}
\newcommand{\Q}{{\mathbb Q}}

\newcommand{\fa}{{\mathfrak a}}
\newcommand{\fp}{{\mathfrak p}}

\newcommand{\fq}{{\mathfrak q}}

\newcommand{\cO}{{\mathcal O}}
\newcommand{\Cl}{\operatorname{Cl}}

\newcommand{\Gal}{\operatorname{Gal}}

\newcommand{\eps}{\varepsilon}

\newtheorem{lem}{Lemma}[section]
\newtheorem{thm}[lem]{Theorem}
\newtheorem{prop}[lem]{Proposition}
\newtheorem{cor}[lem]{Corollary}

\title{Kronecker-Weber via Stickelberger}
\author{Franz Lemmermeyer}
\address{M\"orikeweg 1, 73489 Jagstzell}
\email{hb3@ix.urz.uni-heidelberg.de}
\begin{document}

\maketitle

\begin{abstract}
Nous donnons une nouvelle d\'emonstration du th\'eor\`eme de Kronecker
et Weber fond\'ee sur la th\'eorie de Kummer et le th\'eor\`eme 
de Stickelberger.

In this note we give a new proof of the theorem of Kronecker-Weber
based on Kummer theory and Stickelberger's theorem.
\end{abstract}

\section*{Introduction}

The theorem of Kronecker-Weber states that every abelian extension
of $\Q$ is cyclotomic, i.e., contained in some cyclotomic field. 
The most common proof found in textbooks is based on proofs given
by Hilbert \cite{Hil} and Speiser \cite{Spei}; a routine argument 
shows that it is sufficient to consider cyclic extensions  of 
prime power degree $p^m$ unramified outside $p$, and this special 
case is then proved by a somewhat technical calculation of 
differents using higher ramification groups and an application of 
Minkowski's theorem, according to which every extension of $\Q$ is 
ramified. In the proof below, this not very intuitive part is
replaced by a straightforward argument using Kummer theory
and Stickelberger's theorem.

In this note, $\zeta_m$ denotes a primitive $m$-th root of unity,
and ``unramified'' always means unramified at all finite primes. 
Moreover, we say that a normal extension $K/F$ 
\begin{itemize}
\item is of type $(p^a,p^b)$ if 
      $\Gal(K/F) \simeq (\Z/p^a\Z) \times (\Z/p^b\Z)$;
\item has exponent $m$ if $\Gal(K/F)$ has exponent $m$.
\end{itemize}

\section{The Reduction}

In this section we will show that it is sufficient to prove
the following special case of the Kronecker-Weber theorem
(it seems that the reduction to extensions of prime degree 
is due to Steinbacher \cite{Stein}):

\begin{prop}\label{PExp}
The maximal abelian extension of exponent $p$ that is 
unramified outside $p$ is cyclic: it is the subfield of 
degree $p$ of $\Q(\zeta_{p^2})$. 
\end{prop}

The corresponding result for the prime $p=2$ is easily proved:
 
\begin{prop}\label{PEx2}
The maximal real abelian $2$-extension of $\Q$ with exponent $2$ 
and unramified outside $2$ is cyclic: it is the subfield 
$\Q(\sqrt{2}\,)$ of $\Q(\zeta_8)$.
\end{prop}

\begin{proof}
The only quadratic extensions of $\Q$ that are unramified outside $2$ 
are  $\Q(i)$, $\Q(\sqrt{-2}\,)$, and $\Q(\sqrt{2}\,)$. 
\end{proof}

The following simple observation will be used repeatedly below:

\begin{lem}\label{LC}
If the compositum of two cyclic $p$-extensions $K$, $K'$ 
is cyclic, then $K \subseteq K'$ or $K' \subseteq K$.
\end{lem}

Now we show that Prop. \ref{PExp} implies the corresponding 
result for extensions of prime power degree:

\begin{prop}\label{PCp}
Let $K/\Q$ be a cyclic extension of odd prime power degree $p^m$ 
and unramified outside $p$. Then $K$ is cyclotomic.
\end{prop}

\begin{proof}
Let $K'$ be the subfield of degree $p^m$ in $\Q(\zeta_{p^{m+1}})$. 
If $K'K$ is not cyclic, then it contains a subfield of type $(p,p)$ 
unramified outside $p$, which contradicts Prop. \ref{PExp}. Thus
$K'K$ is cyclic, and Lemma \ref{LC} implies that $K = K'$.
\end{proof}

Next we prove the analog for $p=2$:

\begin{prop}\label{PC2}
Let $K/\Q$ be a cyclic extension of degree $2^m$ and unramified 
outside $2$. Then $K$ is cyclotomic.
\end{prop}

\begin{proof}
If $m = 1$ we are done by Prop. \ref{PEx2}. 
If $m \ge 2$, assume first that $K$ is nonreal. Then
$K(i)/K$ is a quadratic extension, and its maximal
real subfield $M$ is cyclic of degree $2^m$ by Prop. \ref{PEx2}. 
Since $K/\Q$ is cyclotomic if and only if $M$ is, we may assume
that $K$ is totally real.

Now let $K'$ be the the maximal real subfield of  
$\Q(\zeta_{2^{m+2}})$. If $K'K$ is not cyclic, then it contains
three real quadratic fields unramified outside $2$, which
contradicts Prop. \ref{PEx2}. Thus $K'K$ is cyclic, and
Lemma \ref{LC} implies that $K = K'$.
\end{proof}

Now the theorem of Kronecker-Weber follows:
first observe that abelian groups are direct products of cyclic
groups of prime power order; this shows that it is sufficient
to consider cyclic extensions of prime power degree $p^m$. If 
$K/\Q$ is such an extension, and if $q \ne p$ is ramified in 
$K/\Q$, then there exists a cyclic cyclotomic extension $L/\Q$
with the property that $KL = FL$ for some cyclic extension $F/\Q$ 
of prime power degree in which $q$ is unramified. Since $K$ is 
cyclotomic if and only if $F$ is, we see that after finitely
many steps we have reduced Kronecker-Weber to showing that
cyclic extensions of degree $p^m$ unramified outside $p$
are cyclotomic. But this is the content of Prop. \ref{PCp}
and \ref{PC2}.
 
Since this argument can be found in all the proofs based on 
the Hilbert-Speiser approach (see e.g. Greenberg \cite{Gre}
or Marcus \cite{Marc}), we need not repeat the details here.

\section{Proof of Proposition \ref{PExp}}

Let $K/\Q$ be a cyclic extension of prime degree $p$ and 
unramified outside $p$. We will now use Kummer theory to 
show that it is cyclotomic. For the rest of this article,
set $F = \Q(\zeta_p)$ and define $\sigma_a \in G = \Gal(F/\Q)$
by $\sigma_a(\zeta_p) = \zeta_p^a$ for $1 \le a < p$.

\begin{lem}\label{LK1}
The Kummer extension  $L = F(\sqrt[{}^p\,]{\mu}\,)$ is abelian over 
$\Q$ if and only if for every $\sigma_a \in G$ there is a 
$\xi \in F^\times$ such that $ \sigma_a(\mu) = \xi^p \mu^a$.
\end{lem}

For the simple proof, see e.g. Hilbert \cite[Satz 147]{HZB} or
Washington \cite[Lemma 14.7]{Wash}.

Let $K/\Q$ be a cyclic extension of prime degree $p$ and 
unramified outside $p$. Put $F = \Q(\zeta_p)$ and $L = KF$;
then $L = F(\sqrt[{}^p\,]{\mu}\,)$ for some nonzero $\mu \in \cO_F$,
and $L/F$ is unramified outside $p$. 
 
\begin{lem}
Let $\fq$ be a prime ideal in $F$ with $(\mu) = \fq^r\fa$, 
$\fq \nmid \fa$; if $p \nmid r$ and $L/\Q$ is abelian, then 
$\fq$ splits completely in $F/\Q$.
\end{lem}

\begin{proof}
Let $\sigma$ be an element of the decomposition group 
$Z(\fq|q)$ of $\fq$. Since $L/\Q$ is abelian, we must have
$\sigma_a(\mu) = \xi^p \mu^a$. 
Now $\sigma_a(\fq) = \fq$ implies $\fq^r \parallel \xi^p \mu^a$,
and this implies $r \equiv ar \bmod p$; but $p \nmid r$ show 
that this is possible only if $a = 1$. Thus $\sigma_a = 1$, and 
$\fq$ splits completely in $F/\Q$.
\end{proof}

In particular, we find that $(1-\zeta) \nmid \mu$. Since
$L/F$ is unramified outside $p$, prime ideals $\fp \nmid p$
must satisfy $\fp^{bp} \parallel \mu$ for some integer $b$.
This shows that $(\mu) = \fa^p$ is the $p$-th power of some
ideal $\fa$. From $(\mu) = \fa^p$ and the fact that $L/\Q$ 
is abelian we deduce that $\sigma_a(\fa)^p = \fa^{pa}\xi^p$,
where $\sigma_a(\zeta_p) = \zeta_p^a$. Thus 
$\sigma_a(c) = c^a$ for the ideal class $c = [\fa]$ and for
every $a$ with $1 \le a < p$. Now we invoke Stickelberger's 
Theorem (cf. \cite{IR} or \cite[Chap. 11]{Lem}) to show
that $\fa$ is principal:

\begin{thm}
Let $F= \Q(\zeta_p)$; then the Stickelberger element 
$$\theta = \sum_{a=1}^{p-1} a \sigma_a^{-1} \in \Z[\Gal(F/\Q)]$$
annihilates the ideal class group $\Cl(F)$. 
\end{thm}

From this theorem we find that 
$1 = c^\theta =  \prod \sigma_a^{-1}(c)^a = c^{p-1} = c^{-1}$,
hence $c = 1$ as claimed. In particular $\fa = (\alpha)$ is
principal. This shows that $\mu = \alpha^p\eta$ for some unit
$\eta$, hence $L = F(\sqrt[{}^p\,]{\eta}\,)$. Now write
$\eta = \zeta^t\eps$ for some unit $\eps$ in the maximal real
subfield of $F$. Since $\eps$ is fixed by complex conjugation 
$\sigma_{-1}$ and since $L/\Q$ is abelian, we see that
$\zeta^{-t}\eps = \sigma_{-1}(\mu) = \xi^p \mu^{-1}$, hence 
$\zeta^{-t}\eps = \xi^p \zeta^{-t} \eps^{-1}$. 
Thus $\eps$ is a $p$-th power, and we find $\mu = \zeta^t$.
But this implies that $L = \Q(\zeta_{p^2})$, and Prop. \ref{PExp}
is proved.

\medskip

Since every cyclotomic extension is ramified, we get the following
special case of Minkowski's theorem as a corollary:

\begin{cor}
Every solvable extension of $\Q$ is ramified.
\end{cor}

\section*{Acknowledgement}
It is my pleasure to thank the unknown referee for the careful reading
of the manuscript.

\end{document}